\documentclass[10.5pt]{article}

\usepackage[margin= 2 cm, bottom=18mm,footskip=10mm,top=18mm]{geometry}

\usepackage{amsthm}
\usepackage{amsmath}
\usepackage{amssymb}
\usepackage{setspace}
\usepackage{mathtools}
\usepackage{verbatim}
\usepackage{csquotes}
\usepackage{graphicx}
\usepackage[hidelinks]{hyperref}
\usepackage{thm-restate}
\usepackage{cleveref}
\usepackage{graphicx}
\usepackage{enumitem}
\setlist[itemize]{leftmargin=*} 
\setlist[enumerate]{leftmargin=*}
\usepackage{framed}
\usepackage{subcaption}
\usepackage{floatrow}
\usepackage[T1]{fontenc}
\usepackage{mathdots}
\usepackage{xcolor}
\usepackage{diagbox}
\usepackage{colortbl}

\theoremstyle{plain}

\newtheorem{theorem}{Theorem}[section]

\newtheorem{proposition}[theorem]{Proposition}
\newtheorem{lemma}[theorem]{Lemma}

\newtheorem{observation}[theorem]{Observation}

\theoremstyle{definition}

\newtheorem{question}[theorem]{Question}

\newtheorem*{defn*}{Definition}

\newtheorem*{rem}{Remark}

\expandafter\def\expandafter\normalsize\expandafter{%
    \normalsize
    \setlength\abovedisplayskip{4pt}
    \setlength\belowdisplayskip{4pt}
    \setlength\abovedisplayshortskip{4pt}
    \setlength\belowdisplayshortskip{4pt}
}

\usepackage[square,sort,comma,numbers]{natbib}
\newcommand{\calG}{\mathcal{G}}

\newcommand{\Gnp}{\mathcal{G}(n, p)}

\def\eps {\varepsilon}

\newcommand{\cupdot}{\mathbin{\mathaccent\cdot\cup}}
\newcommand{\tc}{\mathrm{tc}}
\newcommand{\kc}{\mathrm{kc}}
\newcommand{\ed}{\mathrm{ed}}
\newcommand{\vt}{\mathrm{vt}}
\newcommand{\intersection}{\mathrm{int}}

\newcommand{\hi}{\mathrm{hi}}
\newcommand{\Hi}{\mathrm{Hi}}

\newcommand{\ceil}[1]{
    \left \lceil #1 \right \rceil
}
\newcommand{\floor}[1]{
    \left \lfloor #1 \right \rfloor
}

\usepackage{tikz}
\usetikzlibrary{calc}
\usetikzlibrary{decorations.pathreplacing}
\usetikzlibrary{positioning,patterns}
\usetikzlibrary{arrows,shapes,positioning}
\usetikzlibrary{decorations.markings}
\def \vx {circle[radius = .25][fill = white, line width=0.5pt]}

\tikzstyle{edge}=[very thick]
\definecolor{bostonuniversityred}{rgb}{0.8, 0.0, 0.0}
\definecolor{arsenic}{rgb}{0.23, 0.27, 0.29}
\tikzstyle{diredge}=[postaction={decorate,decoration={markings,
		mark=at position .55 with {\arrow[scale = 1.5]{stealth};}}}]
\tikzstyle{bidiredge}=[postaction={decorate,decoration={markings,
		mark=at position .7 with {\arrow[scale = 1.5]{stealth};}}}]

\tikzstyle{mdiredge}=[line width = 1 pt, postaction={decorate,decoration={markings,
		mark=at position .625 with {\arrow[scale = 1.5]{stealth};}}}]

\newcommand{\defPt}[3]{
	\def \pt {(#1, #2)}
	\coordinate [at = \pt, name = #3];
}

\newcommand{\fitellipsis}[2] 
{\draw [fill=white]let \p1=(#1), \p2=(#2), \n1={atan2(\y2-\y1,\x2-\x1)}, \n2={veclen(\y2-\y1,\x2-\x1)}
    in ($ (\p1)!0.5!(\p2) $) ellipse [ x radius=\n2/2+0cm, y radius=0.25cm, rotate=\n1];
}

\title{Covering random graphs with monochromatic trees}

\author{
Domagoj Brada\v{c}\thanks{Department of Mathematics, ETH, Z\"urich, Switzerland. Email: \href{mailto:domagoj.bradac@math.ethz.ch} {\nolinkurl{domagoj.bradac@math.ethz.ch}}. Research supported in part by SNSF Grant 200021\_196965.}
 \and
Matija Buci\'c\thanks{School of Mathematics, Institute for Advanced Study and Department of Mathematics, Princeton University, Princeton, 08540 NJ, USA. Email: \href{mailto:matija.bucic@ias.edu} {\nolinkurl{matija.bucic@ias.edu}}.}
}

\date{}

\begin{document}

\maketitle

\begin{abstract}
Given an $r$-edge-coloured complete graph $K_n$, how many monochromatic connected components does one need in order to cover its vertex set? This natural question is a well-known essentially equivalent formulation of the classical Ryser's conjecture which, despite a lot of attention over the last 50 years, still remains open. A number of recent papers consider a sparse random analogue of this question, asking for the minimum number of monochromatic components needed to cover the vertex set of an $r$-edge-coloured random graph $\Gnp$.

Recently, Buci\'{c}, Kor\'{a}ndi and Sudakov established a connection between this problem and a certain Helly-type local to global question for hypergraphs raised about 30 years ago by Erd\H{o}s, Hajnal and Tuza. 
We identify a modified version of the hypergraph problem which controls the answer to the problem of covering random graphs with monochromatic components more precisely. To showcase the power of our approach, we essentially resolve the $3$-colour case by showing that $(\log n / n)^{1/4}$ is a threshold at which point three monochromatic components are needed to cover all vertices of a $3$-edge-coloured random graph, answering a question posed by Kohayakawa, Mendon{\c{c}}a, Mota and Sch{\"u}lke. Our approach also allows us to determine the answer in the general $r$-edge coloured instance of the problem, up to lower order terms, around the point when it first becomes bounded, answering a question of Buci\'c, Kor\'andi and Sudakov.
\end{abstract}

\section{Introduction}
Given a graph $G$ with edges coloured in $r$ colours, how many monochromatic trees, paths or cycles are required to cover the vertices of $G$? Problems of this type were studied extensively by various researchers starting from the 1960`s when Gerencs{\'e}r and Gy{\'a}rf{\'a}s \cite{gerencser1967ramsey} observed that for any $2$-colouring of the edges of the complete graph, all the vertices can be covered using two monochromatic paths. For detailed history on the problems of this type, we refer the reader to a survey by Gy{\'a}rf{\'a}s \cite{gyarfas2016vertex}.

We study the problem of covering an $r$-edge-coloured graph using monochromatic trees. For a graph $G$ and a positive integer $r,$ let $\tc_r(G)$ denote the minimum integer $m$ such that for any $r$-colouring of the edges of $G,$ there exists a collection of $m$ monochromatic trees that cover all vertices of $G.$ As each connected graph has a spanning tree, $\tc_r(G)$ equals the minimum number of monochromatic connected components which cover all vertices. When discussing this problem, we use the terms monochromatic trees and monochromatic components interchangeably. The problem of covering graphs with monochromatic components was first considered by Lov{\'a}sz in 1975 \cite{lovasz-covers} and Ryser in 1970 \cite{henderson} who conjectured that $\tc_r(K_{n}) = r-1.$ That is, given any $r$-colouring of the edges of the complete graph on $n$ vertices, there exists a collection of $r-1$ monochromatic components covering all the vertices. The conjecture was proven for $r \leq 5$ by Tuza \cite{tuza}. On the other hand, it is easy to see that $\tc_r(G) \leq r$ by fixing a particular vertex and choosing the component containing it in every colour.

In recent years, a common theme in combinatorics has been studying sparse random analogues of extremal or Ramsey-type results. For some examples, see e.g. Conlon and Gowers \cite{conlon2016combinatorial} and Schacht \cite{schacht2016extremal}. In line with this theme, Bal and DeBiasio \cite{balpartitioning} initiated the study of covering random graphs with monochromatic components. Following this, Kor{\'a}ndi, Mousset, Nenadov, \v{S}kori\'{c} and Sudakov \cite{cycle-cover} and Lang and Lo \cite{lang} studied a version of this problem in which cycles are used instead of components and Bennett, DeBiasio, Dudek and English \cite{bennett2019large} looked at a similar problem for random hypergraphs. 

In this paper we will focus on the original problem of Bal and DeBiasio \cite{balpartitioning} of covering random graphs with monochromatic components. One of the first natural questions here is, when does $\tc_r(G)$ become bounded? Bal and DeBiasio showed that the answer is when $p$ is somewhere between $\left({\log n}/{n} \right)^{1/r}$ and $\left({\log n}/{n} \right)^{1/(r+1)}.$ This question was subsequently resolved by Buci\'c, Kor\'andi and Sudakov in \cite{bucic2019covering} who determined that the threshold for this property is $\left({\log n}/{n} \right)^{1/r}$. It is worth mentioning here that  $p = (\log n /n)^{1/k}$ is the threshold for any $k$ vertices in $\Gnp$ to have a common neighbour. With this in mind, the above result shows that the threshold for $\tc_r(\Gnp)$ to be finite matches the threshold for any $r$ vertices in $\Gnp$ to have a common neighbour.

Turning to the other side of the spectrum when $p=1,$ we recover the classical problem of Lov\'asz and Ryser and we know that $r$ components suffice. While in this deterministic case the main question was whether $r-1$ components always suffice, in the random one as soon as $p \le 1-\eps,$ for any constant $\eps > 0,$ it is not hard to show that indeed $r$ components are needed. This raises a second very natural question, what is the smallest density $p$ for which we can still ensure that the smallest possible number of components, namely $r$ of them, suffice.

Here Bal and DeBiasio conjectured that this already occurs at the point when $\tc_r(\Gnp)$ becomes finite, in other words that the threshold for $\tc_r(\Gnp) \leq r$ is at the point when any $r$ vertices have a common neighbour (so $p = \left({\log n}/{n}\right)^{1/r}).$ This was proved for $r=2$ by Kohayakawa, Mota and Schacht \cite{kohayakawa2019monochromatic}. However, for $r \geq 3,$ Ebsen, Mota and Schnitzer found a construction showing that w.h.p.\ one needs at least one additional component when $p \ll \left({\log n}/{n}\right)^{1/(r+1)},$ thus disproving the conjecture. It was still widely believed that the conjecture of Bal and DeBiasio is not far from being true, although it was still open how to cover with $r$ components even when the density $p$ is constant. Recently, Buci\'{c}, Kor\'{a}ndi and Sudakov \cite{bucic2019covering} proved that the edge density actually needs to be exponentially larger than conjectured and that this is not far from the truth. In particular, they show that the threshold for $\tc_r(\Gnp) \leq r$ is between the points at which $\frac1{4\sqrt{r}}\cdot 2^{r}$ and $2^r$ vertices have a common neighbour.

The approach in \cite{bucic2019covering} hinges on a close connection it establishes between the tree covering problem and a certain Helly type, local to global problem for hypergraphs which was first considered by Erd\H{o}s, Hajnal and Tuza \cite{erdHos1991local} about 30 years ago. While their approach performs admirably and allows them to get quite good approximate understanding of the behaviour of $\tc_r(\Gnp)$ for any $r$ and $p,$ it turns out to be fundamentally imprecise and hence ill-suited for obtaining more precise results. In this paper we identify, in some sense, the ``correct'' variant of the hypergraph covering problem and establish a connection to the tree covering problem which allows one to obtain significantly more precise results. Before specifying the details in the following subsection we illustrate its performance.

Let us turn back to the problem of determining the threshold at which point $r$ components suffice to cover $\Gnp$. Since the $2$-colour case was resolved completely by Kohayakawa, Mota and Schacht, the $3$-colour case arises as the natural next step. This problem was recently considered by Kohayakawa, Mendon{\c{c}}a, Mota and Sch{\"u}lke \cite{kohayakawa2019covering} who proved that the threshold is at most $\left({\log n}/n \right)^{1/6},$ improving the previous bound of $\left({\log n}/n \right)^{1/8}$ from \cite{bucic2019covering}. In fact, $\left({\log n}/n \right)^{1/6}$ is also the hard limit of the approach of \cite{bucic2019covering} and was implicitly conjectured there to be the answer. On the other hand, the best known example only shows that one needs density at least $\left({\log n}/n \right)^{1/4}$. Given this, Kohayakawa, Mendon{\c{c}}a, Mota and Sch{\"u}lke said it would be very interesting to determine the correct threshold. Using our new connection result we answer their question and completely determine the threshold. Perhaps surprisingly, it turns out one can do much better than in either of \cite{kohayakawa2019covering,bucic2019covering} and the lower bound of $\left({\log n}/n \right)^{1/4}$ due to Ebsen, Mota and Schnitzer turns out to be the truth.

\begin{restatable}{theorem}{thmthreerestate} \label{thm:three_threshold}
  The threshold for $\tc_3(\Gnp) \le 3$ is equal to $\left({\log n}/{n}\right)^{1/4}$. 
\end{restatable}

Let us now turn back to the point at which $\tc_r(\Gnp)$ becomes finite, so when any $r$ vertices have a common neighbour. Buci\'{c}, Kor\'{a}ndi and Sudakov showed that at this point $r^2(1-o(1)) \le \tc_r(\Gnp)\le (3r-2)r,$\footnote{Exceptionally here the $o(1)$ term is in terms of $r$ and not $n$.} which determined it up to a constant factor, and asked for more precise bounds. Our connection result allows us to determine it up to lower order terms answering a question of Buci\'{c}, Kor\'{a}ndi and Sudakov. 

\begin{theorem}\label{prop:begin-range}
	For any $r \geq 2$ there is a constant $C = C(r)$ such that if $p > C\left(\frac{\log n}{n} \right)^{1/r},$ then w.h.p. $G \sim \Gnp$ satisfies $\tc_r(G) \leq r^2.$
\end{theorem}
\noindent It is worth pointing out that, unlike for us, in order to establish their weaker upper bound for this regime, \cite{bucic2019covering} needed to work directly with random graphs since their connection result does not apply here. This highlights the fact that our connection result allows one to study $\tc_r(\Gnp)$ in terms of the hypergraph problem for any choice of $p$.

\subsection*{The connection to cover number problem for $r$-partite hypergraphs}
Buci\'{c}, Kor{\'a}ndi and Sudakov \cite{bucic2019covering} related the problem of covering a random graph with monochromatic components to the following local-global extremal problem for hypergraph covers, first considered by Erd\H{o}s, Hajnal and Tuza \cite{erdHos1991local} about 30 years ago. 
Suppose $H$ is an $r$-uniform hypergraph in which every $k$ edges can be covered by a set of at most $r$ vertices. How large can the cover number of $H$ be? This natural problem was later studied by Erd\H{o}s, Fon-Der-Flaass, Kostochka and Tuza; Fon-Der-Flaass, Kostochka and Woodall and Kostochka \cite{erdos1992small, kostochka2002transversals,fon1999transversals}. In the specific instance needed for the connection, Buci\'{c}, Kor{\'a}ndi and Sudakov obtain very good bounds which allow them to upper bound the probability needed to guarantee $\tc_r(\Gnp) \le t$. They discovered that, quite remarkably, the connection actually goes two ways provided one considers the following closely related $r$-partite variant of the problem of Erd\H{o}s, Hajnal and Tuza. Suppose $H$ is an $r$-partite $r$-uniform hypergraph in which every $k$ edges have a transversal cover, defined as a cover containing one vertex per part. How large can the cover number of $H$ be?

Using their connection, one can at best place the threshold for $\tc_r(\Gnp) \le t$ between thresholds for any $k$ vertices to have a common neighbour and any $O(k)$ vertices to have a common neighbour, where $k$ is given by the solution of an appropriate instance of the $r$-partite cover number problem.
We identify a further modification of the hypergraph problem, the solution of which places the threshold between those for $k$ and $k+1$ vertices to have a common neighbour, for an appropriate value of $k.$
Our modification asks for us to be able to fix a transversal cover for any $k$ edges in such a way that any pair of these fixed covers intersect.

\begin{restatable}{defn*}{defnintkcoversrestate}
Given an $r$-partite $r$-graph $H,$ we say it has \emph{intersecting $k$-covers} if there exists a function\linebreak $\phi \colon E(H)^k \rightarrow E(H)$ such that $\phi(T)$ is a cover of $T,$ for any $k$-tuple of edges $T,$ and $\phi(T) \cap \phi(T') \neq \emptyset,$ for any $k$-tuples of edges $T, T'.$
\end{restatable}

We are now interested in the maximum possible cover number of an $r$-partite $r$-uniform hypergraph with intersecting $k$-covers. To state our result formally, it will be more convenient to define the inverse function of the answer to this question. For positive integers $r,t,$ we define $\kc_r(t)$ to be the minimum value of $k$ such that every $r$-partite $r$-graph with intersecting $k$-covers has cover number at most $t.$ We show that the threshold for $\tc_r(\Gnp) \leq t$ is between the points when any $k=\kc_r(t)$ and any $k+1$ vertices are likely to have a common neighbour. 

\begin{restatable}{theorem}{thmuptoonerestate} \label{thm:upto1}
  Suppose $r,t$ are positive integers and $k = \kc_r(t) < \infty.$ Then, there exist constants $C,c > 0$ such that for $G \sim \Gnp,$
  \begin{enumerate}[label=\alph*)]
    \item if $p < c \left( \frac{\log n}{n} \right)^{\frac{1}{k}},$ then w.h.p. $\tc_r(G) > t,$ and
    \item if $p > C \left( \frac{\log n}{n} \right)^{\frac{1}{k+1}},$ then w.h.p. $\tc_r(G) \leq t.$
  \end{enumerate}
\end{restatable}

This result comes very close to showing that indeed the threshold for $\tc_r(\Gnp) \le t$ matches the one for any $k$ vertices in $\Gnp$ to have a common neighbour, for some $k$ depending only on $r$ and $t$. It turns out that, at least in certain cases, this is in some sense the best one can hope for since, for example, the threshold for $\tc_3(\Gnp)\le 4$ is somewhere between $n^{-1/3+\eps}$ and $n^{-1/4-\eps},$ for some $\eps>0$ (see \cite{thesis} and the end of Section 5).

With the goal of being able to determine the thresholds precisely we identify a further refinement of the hypergraph covering problem which allows for an even stronger connection result. In particular, this is behind our proof of \Cref{thm:three_threshold}.

\textbf{Notation.} We say that a function $p^*:\mathbb{N} \to [0,1]$ is the \emph{threshold} for an increasing property of graphs if there are constants $C,c>0$ such that $p > C p^*(n)$ implies $\Gnp$ satisfies the property w.h.p.\ and $p<cp^*(n)$ implies $\Gnp$ does not satisfy the property w.h.p. We note that throughout the paper we will consider $r,$ the number of colours (or the number of parts in the hypergraphs setting), to be fixed.

\medskip
\textbf{Organisation of the paper.} 
The paper is organised as follows. In Section 2, we state several known results used in our proofs as well as a few easy statements about random graphs. In Section 3, we develop the connection between covering random graphs with monochromatic components and covering hypergraphs. In particular, we prove Theorem~\ref{thm:upto1} as well as a further refinement which we later use when discussing the case $r=3.$ Finally, in Section 4 we consider the $3$-coloured case of the tree covering problem for random graphs, where we prove Theorem~\ref{thm:three_threshold}.

\section{Preliminaries}
We begin with several basic definitions about (hyper)graphs. The sets of vertices and edges of a (hyper)graph $G$ are denoted by $V(G)$ and $E(G),$ respectively. An \emph{$r$-uniform hypergraph} or an \emph{$r$-graph} is a hypergraph with all edges of size $r.$ An $r$-graph $H$ is called $r$-partite if there exists a partition of its vertex set $V(H) = V_1 \cupdot V_2 \cupdot \dots \cupdot V_r$ such that for all $e \in E(H), \, 1 \leq i \leq r$ we have $|e \cap V_i| = 1$ and we use $e^i$ to denote the vertex of $e$ in part $V_i.$ A (vertex) cover of a hypergraph is a set of vertices intersecting every edge. The (vertex) cover number of a hypergraph $H,$ denoted by $\tau(H),$ is the minimum size of a vertex cover of $H.$

Given a graph $G$ and a set of vertices $A \subseteq V(G),$ we denote 
\[ N(A) = A \cup \{ v \in V(G) \; \vert \; \exists u \in A, \, uv \in E(G) \}. \]
Given $k$ sets of vertices $A_1, \dots, A_k,$ we define $N(A_1, \dots, A_k) = \bigcap_{i=1}^k N(A_i).$

We say that a random graph $G \sim \Gnp$ satisfies a certain property with high probability (w.h.p.) if the probability that $G$ satisfies the property tends to $1$ as $n \to \infty$. We now list a number of mostly standard and easy to verify properties of random graphs which we will use in our proofs.
\begin{restatable}{lemma}{nolargeemptybipartiterestate}[e.g. \cite{cycle-cover}, Lemma 3.8]

  \label{lemma:no_large_empty_bipartite}
  Let $p = p(n) > 0.$ The random graph $G \sim \Gnp$ satisfies the following property w.h.p.: for any two disjoint subsets $A, B \subseteq V(G)$ of size at least $\frac{10 \log n}{p},$ there is an edge between $A$ and $B.$
\end{restatable}

\begin{restatable}{lemma}{commonneighrestate}[\cite{balpartitioning}, Lemma 6.1.]
  \label{lemma:common_neigh_k_1}
  Let $k \geq 1$ be an integer and let $p > C \left(\frac{\log n}{n}\right)^{\frac{1}{k}}$ for some large enough constant $C.$ Then in $G \sim \Gnp,$ w.h.p. any $k-1$ vertices have a common neighbourhood of size at least $\frac{C}{2} \cdot \frac{\log n}{p}.$
\end{restatable}

\begin{restatable}{lemma}{largeneighsrestate}[\cite{bucic2019covering}, Lemma 2.3]
  \label{lemma:large_neigh}
	Let $k \geq 1$ be an integer and let $C = C(k) > 0$ be a large enough constant. If $p > C\left(\frac{\log n}{n}\right)^{\frac{1}{k}},$ the random graph $G \sim \Gnp$ satisfies the following w.h.p.: for any subset $U \subseteq V(G), \, 1 \leq |U| \leq \frac{1}{p}$ and any $k-1$ vertices $v_1 \dots, v_{k-1} \in V,$ 
\[ \label{k_neigh_property} \left|N\left(U, \{v_1\}, \{v_2\}, \dots \{v_{k-1}\}\right)\right| \geq \frac{C \log n}{6} \cdot |U|. \]
\end{restatable}

\begin{lemma}\label{lemma:random_for_lb}[\cite{balpartitioning}, Lemma 6.4 (ii)]
Let $k\ge 2$ and $m$ be integers. There is a constant $c=c(m,k)>0$ such that for $G\sim \Gnp$ with $p\le c\left(\frac{\log n}{n}\right)^{1/k},$ w.h.p.\ there is an independent set $S$ in $G$ of size $m$ such that no $k$ vertices in $S$ have a common neighbour in $G$.
\end{lemma}

We note that this lemma came with the assumption $m>k$ in \cite{balpartitioning}, however for $m \le k$ one can simply apply the lemma with $m=k+1$ to find an independent set of size $k+1$ with the desired property and take any $m$ vertex subset as our desired $S$.

We finish this section with two classical extremal results.
\begin{theorem}[Bollob\'{a}s \cite{bollobas1965generalized}]
  \label{thm:bollobas}
  Let $H$ be an $r$-graph with cover number at least $t.$ Then, there exists a subgraph of $H$ with cover number $t$ and at most $\binom{r + t - 1}{r}$ edges.
\end{theorem}

\begin{theorem}[K\H{o}nig's theorem \cite{konig}] \label{thm:konig}
  The size of a largest matching in a bipartite graph equals its cover number.
\end{theorem}

\section{The connection to hypergraph covering}\label{sec:connection}


Given a graph $G,$ an $r$-edge-colouring $c,$ and a subset of vertices $W,$ we define the auxiliary $r$-partite $r$-graph $H = H(G, W, c)$ and two functions: $\ed = \ed(G, W, c),$ which maps vertices of $G$ to \emph{edges} of $H$; and $\vt = \vt(G, W, c)$ which maps monochromatic components of $G$ under $c$ to \emph{vertices} of $H.$ Formally,
\begin{align*}
  \ed &\colon V(G) \rightarrow E(H),\\
  \vt &\colon \big\{ C \; \vert \; C \text{ is a monochromatic component of } G \text{ under } c \big\} \rightarrow V(H).
\end{align*}
Each part of $H$ corresponds to one of the $r$ colours. For every monochromatic component $C$ (possibly consisting of a single vertex) in colour $i, \, 1 \leq i \leq r,$ such that $C \cap W \neq \emptyset,$ $H$ contains a vertex $\vt(C)$ in part $i.$ Additionally, in each part $i, 1 \leq i \leq r,$ $H$ contains a special vertex $v_i^*$ which corresponds to all other monochromatic components in colour $i$ (those which do not contain any vertex in $W$). In other words, if $C \cap W = \emptyset,$ then $\vt(C) = v_i^*.$
For each vertex $u \in V(G),$ we put an $r$-edge $\ed(u)$ defined as 
\[ \ed(u) = \{ \vt(C_i) \; \vert \; C_i \text{ is the monochromatic component in colour } i \text{ containing } u\}. \]
Note that the function $\ed$ is surjective ($H$ does not include redundant edges), but not necessarily injective.

We use shorthands $H(G, c) = H(G, V(G), c), \, \ed(G, c) = \ed(G, V(G), c), \, \vt(G, c) = \vt(G, V(G), c).$

\begin{figure}
    \centering
    \begin{tikzpicture}[line width=2pt]
    \defPt{-10}{0}{x1}
    \defPt{-10}{1.5}{x2}
    \defPt{-10}{3}{x3}
    \defPt{-10}{4.5}{x4}
    
    \defPt{-8}{0}{y1}
    \defPt{-8}{1.5}{y2}
    \defPt{-8}{3}{y3}
    \defPt{-8}{4.5}{y4}
    
    \defPt{-2}{1}{r1}
    \defPt{-3.2}{1.55}{bstar}
    \defPt{-1.75}{-0.2}{gdash}
    
    \defPt{-0}{4.464}{b1}
    \defPt{-1.1}{5}{rdash}
    \defPt{1.1}{5}{gstar}
    
    \defPt{2}{1}{g1}
    \defPt{3.2}{1.55}{bdash}
    \defPt{1.75}{-0.2}{rstar}
    
    \defPt{-2}{3.2}{g2}
    \defPt{0}{0.1}{b2}
    \defPt{2}{3.2}{r2}
    
    \draw[red] (x4) -- (y1);
    \draw[red] (x4) -- (y2);
    \draw[red] (x4) -- (y3);
    
    \draw[blue] (x3) -- (y4);
    \draw[blue] (x3) -- (y2);
    \draw[blue] (x3) -- (y1);
    
    \draw[green] (x2) -- (y4);
    \draw[green] (x2) -- (y3);
    \draw[green] (x2) -- (y1);
    
    \draw[red] (x1) -- (y4);
    \draw[blue] (x1) -- (y3);
    \draw[green] (x1) -- (y2);
    
    \draw[blue] (y1) -- (y2); 
    \draw[red] (y2) -- (y3); 
    \draw[green] (y4) -- (y3); 
    \draw[blue] (y2)  to [out=60,in=-60] (y4); 
    \draw[red] (y1)  to [out=60,in=-60] (y3); 
    \draw[green] (y1)  to [out=50,in=-50] (y4); 
    
    \node at ($(x1)-(2,0)$) {$\ed(x_1)=\textcolor{red}{r_2}\textcolor{blue}{b_2}\textcolor{green}{g_2}$};
    
    \node at ($(x2)-(2,0)$) {$\ed(x_2)=\textcolor{red}{r^*}\textcolor{blue}{b'}\textcolor{green}{g_1}$};
    
    \node at ($(x3)-(2,0)$) {$\ed(x_3)=\textcolor{red}{r'}\textcolor{blue}{b_1}\textcolor{green}{g^*}$};
    
    \node at ($(x4)-(2,0)$) {$\ed(x_4)=\textcolor{red}{r_1}\textcolor{blue}{b^*}\textcolor{green}{g'}$};
    
    \node at ($(y1)+(2.5,0)$) {$\textcolor{red}{r_1}\textcolor{blue}{b_1}\textcolor{green}{g_1}=\ed(y_1)$};
    
    \node at ($(y2)+(2.5,0)$) {$\textcolor{red}{r_1}\textcolor{blue}{b_1}\textcolor{green}{g_2}=\ed(y_2)$};
    
    \node at ($(y3)+(2.5,0)$) {$\textcolor{red}{r_1}\textcolor{blue}{b_2}\textcolor{green}{g_1}=\ed(y_3)$};
    
    \node at ($(y4)+(2.5,0)$) {$\textcolor{red}{r_2}\textcolor{blue}{b_1}\textcolor{green}{g_1}=\ed(y_4)$};
    
    \foreach \i in {1,...,4}
    {
        \draw[line width=0.5pt]  (x\i) \vx;
        \node at (x\i) {\small{$x_{\i}$}};
        \draw[line width=0.5pt]  (y\i) \vx;
        \node at (y\i) {\small{$y_{\i}$}};
    }
    
    \foreach \i in {r1,b1,g1,r2,b2,g2,bstar,gdash,gstar,rdash,rstar,bdash}
    {
        \draw (\i) \vx;        
    }
    \node at (r1) {\small{$\textcolor{red}{r_1}$}};
    \node at (r2) {\small{$\textcolor{red}{r_2}$}};
    \node at (b1) {\small{$\textcolor{blue}{b_1}$}};
    \node at (b2) {\small{$\textcolor{blue}{b_2}$}};
    \node at (g1) {\small{$\textcolor{green}{g_1}$}};
    \node at (g2) {\small{$\textcolor{green}{g_2}$}};
    \node at (rdash) {\small{$\textcolor{red}{r'}$}};
    \node at (rstar) {\small{$\textcolor{red}{r^*}$}};
    \node at (bdash) {\small{$\textcolor{blue}{b'}$}};
    \node at (bstar) {\small{$\textcolor{blue}{b^*}$}};
    \node at (gdash) {\small{$\textcolor{green}{g'}$}};
    \node at (gstar) {\small{$\textcolor{green}{g^*}$}};
    
    \draw[rounded corners=5mm, fill=black!10!] (r1) -- (b1) -- (g1) -- cycle;
    \draw[rounded corners=5mm, fill=black!10!] ($(r1)+(-0.15,-0.15)$) -- (bstar) -- (gdash) -- cycle;
    \draw[rounded corners=5mm, fill=black!10!] ($(b1)+(0,0.15)$) -- (gstar) -- (rdash) -- cycle;
    \draw[rounded corners=5mm, fill=black!10!] ($(g1)+(0.15,-0.15)$) -- (rstar) -- (bdash) -- cycle;
    
    \draw[rounded corners=5mm, fill=black!10!] ($(r1)+(-0.15,0)$) -- ($(b1)+(-0.1,0.1)$) -- ($(g2)+(0.1,-0.05)$) -- cycle;
    \draw[rounded corners=5mm, fill=black!10!] ($(g1)+(0.15,0)$) -- ($(b1)+(0.1,0.1)$) -- ($(r2)+(-0.1,-0.05)$) -- cycle;
    \draw[rounded corners=5mm, fill=black!10!] ($(r1)+(-0.05,-0.15)$) -- ($(b2)+(0,0.15)$) -- ($(g1)+(0.05,-0.15)$) -- cycle;
    \draw[rounded corners=5mm, fill=black!20!, fill opacity=0.5] (r2) -- (b2) -- (g2) -- cycle;    
        
    \end{tikzpicture}
    \caption{An example of a coloured graph and its auxiliary hypergraph $H(G,c)$. Here $\textcolor{red}{r_1},\textcolor{red}{r_2},\textcolor{blue}{b_1},\textcolor{blue}{b_2},\textcolor{green}{g_1},\textcolor{green}{g_2}$ denote non-trivial monochromatic components, $\textcolor{red}{r^*},\textcolor{blue}{b^*},\textcolor{green}{g^*},\textcolor{red}{r'},\textcolor{blue}{b'},\textcolor{green}{g'}$ denote components consisting of a single isolated vertex. When discussing $H(G,c)$, we think of its vertex set as simply the set of monochromatic components and the function $\vt$ is the identity function allowing us to highlight which setting we are working in. In case of $H(G,W,c)$ it serves another purpose, namely mapping any connected component disjoint from $W$ to our artificial collapsed component $v_i^*$. For example if in the above example $W=\{y_1,y_2,y_3,y_4\}$ then components $\textcolor{red}{r^*},\textcolor{red}{r'}$ collapse to the same component as do $\textcolor{blue}{b^*},\textcolor{blue}{b'}$ and $\textcolor{green}{g^*},\textcolor{green}{g'}$.}
    \label{fig:auxilliary}
\end{figure}

The definition of $H(G, c)$ is behind the connection established in \cite{bucic2019covering}. The introduction of $W$, which one should think of as the set of ``relevant'' vertices is one of the key new ingredients which allow us to establish our new connection result. It captures in some sense the idea that everything ``interesting'' with $H(G,c)$ actually happens because of a small, constant number of edges. In the language of random graphs this means that a relatively small set of vertices prevents one from covering the whole vertex set with few components. 

Let us give some intuition behind this definition. First consider $(H, \ed, \vt) = (H, \ed, \vt)(G, c).$ Every vertex $v \in G$ maps to an edge $\ed(v)$ containing the vertices which correspond to the $r$ monochromatic components (including singleton components) which contain $v.$ Consider some monochromatic component $C$ of $G$ and the corresponding vertex $\vt(C)$ in $H.$ The edges which contain $\vt(C)$ correspond to vertices of $G$ in $C.$ Hence, it is easy to see that covering $G$ with monochromatic components is equivalent to covering the edges of $H$ with vertices.

Now, consider some $W \subseteq V$ and $(H, \ed, \vt) = (H, \ed, \vt)(G, W, c).$ We think of $H$ in the following way. The monochromatic components of $G$ which intersect $W$ are left as they are and each such monochromatic component $C$ is mapped to a unique vertex $\vt(C).$ For $i \in [r],$ the monochromatic components in colour $i$ which do not intersect $W$ are thought of as a single connected component, so they are mapped to the same vertex $v_i^*.$ Therefore, $H(G, W, c)$ is obtained from $H(G, c)$ by conflating some vertices, so it is easier to cover. However, we will choose a set $W$ which still requires many monochromatic components to cover. As the monochromatic components intersecting $W$ are left intact when building $H(G, W, c),$ it will also require many vertices to be covered. \Cref{thm:bollobas} will allow us to choose $W$ to be a small set (of constant size) while still having $\tau(H(G, W, c))$ large.

The following lemma formalises these ideas and collects the properties of $H(G,W,c)$ we are going to use throughout. The first part of the lemma already appears in \cite{bucic2019covering}.
\begin{lemma} \label{lemma:connection}
   Given an $r$-edge-colouring $c$ of $G,$ the following statements hold.
    \begin{enumerate}[label=\alph*)]
      \item \label{whole_graph} The number of components monochromatic in $c$ needed to cover $G$ equals $\tau\left(H(G, c)\right),$ namely $\tc_r(G) \geq \tau(H(G, c)).$
      \item \label{subgraphs_smaller_cover} For any $W \subseteq V(G),$ we have $\tau(H(G, W, c)) \leq \tau(H(G, c)).$
      \item \label{small_subgraph} If $\tau(H(G, c)) \geq s,$ there exists $W \subseteq V(G)$ of size at most $N_{r,s} = \binom{r-1 + s}{r}$ such that $\tau(H(G, W, c)) \geq s.$
      \item \label{neighbours_intersecting} Let $W \subseteq V(G)$ be arbitrary. If $uv \in E(G),$ then, for $\ed = \ed(G, W, c),$ the edges $\ed(u)$ and $\ed(v)$ intersect.
      \item \label{large_subset} Let $W, A \subseteq V(G)$ be arbitrary and consider $H = H(G, W, c), \ed = \ed(G, W, c).$ Then there exists an $f \in E(H)$ such that the number of $u \in A$ for which $\ed(u)=f$ is at least $\frac{|A|}{(|W| + 1)^r}.$ 
   \end{enumerate}
\end{lemma}

\begin{proof}
  We prove these claims in order:
  \begin{enumerate}[label={\alph*})]
    \item Let $C_1, \dots, C_q$ be monochromatic components covering $V(G).$ Then, the vertices $\vt(C_1), \dots, \vt(C_q)$ form a cover of $H(G, c).$ Indeed, suppose that for some $u \in V(G),$ the edge $\ed(u)$ is not covered. Then, $u$ does not belong to any of the monochromatic components $C_1, \dots, C_q,$ a contradiction. Hence, $\tau\left(H(G, c)\right) \leq \tc_r(G).$

      On the other hand, suppose $H(G, c)$ can be covered using $t$ vertices. Observe that the vertices $v_i^*, 1 \leq i \leq r$ are isolated in $H(G, c)$ so we may suppose that $H$ is covered by vertices $\vt(C_1), \dots, \vt(C_t)$ for some monochromatic components $C_1, \dots, C_t$ of $G.$ We show these monochromatic components cover $V(G).$ Suppose some vertex $u \in V(G)$ is not covered. Then, the corresponding edge $\ed(u)$ does not contain any of the vertices $\vt(C_1), \dots, \vt(C_t),$ contradicting our assumption.
    \item Let $W \subseteq V(G)$ be arbitrary. Denote $H = H(G, c), \ed = \ed(G, c), \vt = \vt(G, c)$ and $H_W = H(G, W, c), \ed_W = \ed(G, W, c), \vt_{W}= \vt(G,W,c).$ Suppose $\tau(H) = s$ and consider a minimum cover $X = \{u_1, \dots, u_s\}$ of $H.$ Recall that in $H,$ the vertices $v_i^*, \, i \in [r],$ are isolated, so are not in $X$. Recall that for a monochromatic component $C$ of colour $i$, either $\vt_W(C)=\vt(C)$ (in case $C \cap W \neq \emptyset)$ or $\vt_W(C)=v_i^{*}$. One might wish to think of this as saying that we obtain $H_W$ from $H$ by ``collapsing'' a number of vertices into the $v_i^*$ of their part. 
    Let us now define the modified cover vertices $u_1',\ldots, u_s'$ as follows. Given $u_i$ let $C$ be the monochromatic component corresponding to it, so $\vt(C)=u_i$ and suppose it is of colour $j$. We then set $u_i'=u_i,$ if $C \cap W \neq \emptyset$ (so in case $u_i$ has not collapsed into $v_j^{*}$) and we set $u_i'=v_j^*$, otherwise (i.e.\ if $u_i$ collapsed into $v_j^{*}$). It is not hard to see that if $u_i \in \ed(v),$ then $u_i'\in \ed_W(v)$ so $\{u_1', \dots, u_s'\}$ makes a cover of $H_W$ of size at most $s = \tau(H)$, as desired. 
    
    \item By Theorem \ref{thm:bollobas}, in $H(G, c)$ there exists a set of edges $F$ of size at most $\binom{r-1 + s}{r}$ which cannot be covered by $s - 1$ vertices. Let $W$ be a set of vertices of $G$ obtained by choosing one element of $\ed^{-1}(\{f\})$ for each $f \in F$. Since, $F \subseteq H(G, W, c),$ it follows that $\tau(H(G, W, c)) \geq s.$
    
    \item Let $W \subseteq V(G)$ be arbitrary, let $\ed_W = \ed(G, W, c),$ and suppose $uv \in E(G).$ If $uv$ was coloured $j$ under $c,$ then $u, v$ lie in the same monochromatic component in colour $j.$ It follows that $\ed_W(u)$ and $\ed_W(v)$ contain the same vertex in part $j$.
    \item Observe that, by definition, $H(G, W, c)$ has at most $|W| + 1$ vertices in each of its parts. Therefore, the total number of different edges in $H$ is at most $(|W| + 1)^r.$ Taking the edge $f$ which appears most frequently among $\ed(u), u \in A,$ the claim follows.
  \end{enumerate}
  \vspace{-0.85cm}
\end{proof}

In order to prove lower bounds on $\tc_r(G)$ with $G \sim \Gnp,$ we find a fixed $r$-partite $r$-graph $H_0$ with large cover number $\tau(H_0)$ and show that, for an appropriate value of $p,$ w.h.p. we can find an $r$-edge-coloring $c$ of $G$ such that $H(G, c) \supseteq H_0.$ The following lemma will be a useful tool for obtaining such a colouring.
\begin{lemma} \label{lemma:colouring_lb}
  Let $G$ be a graph and let $H_0$ be an $r$-partite $r$-graph. Suppose we are given a function $\ed_0 \colon V(G) \rightarrow E(H_0)$ such that for any edge $uv \in E(G),$ $\ed_0(u) \cap \ed_0(v) \neq \emptyset.$ If $\ed_0$ is surjective, then $\tc_r(G) \geq \tau(H_0).$
\end{lemma}
\begin{proof}
  First we describe the colouring $c$ such that $G$ cannot be covered by fewer than $\tau(H_0)$ monochromatic components. Consider an arbitrary edge $uv \in E(G).$ By assumption, $\ed_0(u) \cap \ed_0(v) \neq \emptyset$ and let $j \in [r]$ be an arbitrary index such that $\ed_0(u)^j = \ed_0(v)^j,$ where we remind the reader that we use $e^j$ to denote the vertex of $e$ in part $j.$ Then set $c(uv) = j.$

  Now, consider $(H, \ed, \vt) = (H, \ed, \vt)(G, c).$ We claim that for any $u, v \in V(G),$ and any $j \in [r]$:
  \begin{equation} \label{refined_comps}
    \ed(u)^j = \ed(v)^j \implies \ed_0(u)^j = \ed_0(v)^j.
  \end{equation}
  Suppose this is not the case and let $u, v \in V(G), \, j \in [r]$ be such that $\ed(u)^j = \ed(v)^j,$ but $\ed_0(u)^j \neq \ed_0(v)^j.$ Because $\ed(u)^j = \ed(v)^j,$ there exists a monochromatic path $(u = w_1, w_2, \dots, w_q = v)$ with all edges in colour $j.$ Since $\ed_0(u)^j \neq \ed_0(v)^j,$ there is an index $1 \leq i \leq q-1$ such that $\ed_0(w_i)^j \neq \ed_0(w_{i+1})^j.$ However, by construction, $c(w_i, w_{i+1}) \neq j,$ a contradiction.

  Finally, we show that $\tau(H) \geq \tau(H_0),$ which together with  Lemma~\ref{lemma:connection}~part~\ref{whole_graph} implies $\tc_r(G) \geq \tau(H) \geq \tau(H_0),$ as desired. Let $X = \{x_1, \dots, x_s\}$ be a cover of $H.$ By definition, for each $i \in [s],$ we can fix some $v_i \in V(G)$ and $j_i \in [r],$ such that $x_i = \ed(v_i)^{j_i}.$ Now, let $y_i = \ed_0(v_i)^{j_i}$ for $i \in [s].$ We show that $Y = \{y_1, \dots, y_s\}$ is a cover of $H_0.$ Consider an arbitrary vertex $u \in V(G).$ Since $X$ is a cover of $H,$ there exists some $i \in [s]$ such that $\ed(u)^{j_i} = x_i = \ed(v_i)^{j_i}.$ By \eqref{refined_comps}, we have $\ed_0(u)^{j_i} = \ed_0(v_i)^{j_i} = y_i,$ so $\ed_0(u)$ is covered by $Y.$ Because $\ed_0$ is surjective, it follows that $Y$ is a cover of $H_0,$ thus proving $\tau(H) \geq \tau(H_0).$
\end{proof}

The previous two lemmas suggest that in order to study $\tc_r(G),$ it is crucial to understand the properties $H = H(G, c)$ must have w.h.p. for any $r$-edge-colouring $c.$ Suppose we want to find an upper bound on the threshold for the property $\tc_r(G) \leq t,$ for some fixed values $r, t.$ Then, Lemma~\ref{lemma:connection}~part~\ref{small_subgraph} essentially allows us to assume that $H(G, c)$ is of bounded size.
If $p \gg \left(\frac{\log n}{n}\right)^{1/k},$ then w.h.p. every $k$ vertices in $G \sim \Gnp$ have a common neighbour. This implies that every $k$ edges in $H$ have a transversal cover which is also an edge of $H.$ Buci\'{c}, Kor\'{a}ndi and Sudakov \cite{bucic2019covering} observed that this property imposes a lot of structure on $H$ and can be used to bound its cover number and hence also the tree cover number of $G.$ 

We further refine this argument in the following way. If $p \gg \left(\frac{\log n}{n}\right)^{1/k},$ not only do every $k$ vertices in $G$ have a common neighbour, they have a common neighbourhood of size roughly $np^k.$ Since we may assume that $H$ has bounded size, for any $k$-tuple $T$ of vertices in $G,$ we can obtain a subset $U(T)$ of size $\Omega(np^k)$ of the common neighbourhood of $T,$ with all vertices $v \in U(T)$ having the same $\ed(v),$ using \Cref{lemma:connection}~part~\ref{small_subgraph}. If $p$ is large enough, w.h.p. there is an edge $vu$ between $U(T)$ and $U(T')$ for any two $k$-tuples of vertices, $T, T'.$ This implies, by \Cref{lemma:connection}~part~\ref{neighbours_intersecting}, that $\ed(v)$ and $\ed(u)$ intersect. Since all vertices $w \in U(T)$ have the same value of $\ed(w),$ we can set $\phi(T):=\ed(v)$, which by the fact $v$ was a common neighbour of vertices in $T$ and \Cref{lemma:connection}~part~\ref{neighbours_intersecting} implies $\phi(T)$ is a transversal cover of the $k$ edges in $H(G,c)$ corresponding to $T$. Since $T$ was arbitrary this collection of edges in $H$ is also arbitrary and we have shown that $\phi(T) \cap \phi (T') \neq \emptyset$ for any $T,T'$, so any two of these covers need to intersect. In particular, this shows that $H(G,c)$ must satisfy the intersecting $k$-cover property we defined in the introduction and which we repeat here for convenience.

\defnintkcoversrestate*

For positive integers $r \leq k,$ we define $\hi_r(k)$ to be the maximum possible cover number of an $r$-partite $r$-graph with intersecting $k$-covers. Observe that $\hi_r(k)$ is finite, indeed if an $r$-partite $r$-graph has intersecting $k$-covers, it cannot have a matching of size $2r$, so its cover number is at most $r(2r-1).$ 

The above discussion leads us to the following result.

\begin{lemma} \label{lemma:upper_bound_simple}
  Let $2 \leq r \leq k$ be fixed. There exists a constant $C = C(r, k) > 0$ such that the random graph $G \sim \calG(n, p),$ with $p > C \left(\frac{\log n}{n}\right)^{1 / (k+1)},$ satisfies w.h.p. $\tc_r(G) \leq \hi_r(k)$.
\end{lemma}
\begin{proof}
 Let $C$ be large enough and assume that $G$ satisfies the properties given by Lemmas \ref{lemma:no_large_empty_bipartite} and \ref{lemma:common_neigh_k_1}, namely there exists an edge between any two disjoint sets of at least $\frac{10 \log n}{p}$ vertices, and any $k$ vertices have a common neighbourhood of size at least $\frac{C}{2} \cdot \frac{\log n}{p}.$ Let $s = \hi_r(k).$ Suppose $\tc_r(G) \geq s+1$ and let $c$ be an $r$-edge-colouring such that $G$ cannot be covered by $s$ monochromatic components. By Lemma \ref{lemma:connection}~part~\ref{whole_graph}, $\tau(H(G, c)) \geq s+1$ and by Lemma~\ref{lemma:connection}~part~\ref{small_subgraph} there exists a set of vertices $W \subseteq V(G)$ such that $|W| \leq N_{r,s+1}$ and $\tau(H(G, W, c)) \geq s+1.$ We show that $H(G, W, c)$ has intersecting $k$-covers, thus contradicting the definition of $\hi_r(k).$

  Denote $(H, \ed, \vt) = (H, \ed, \vt)(G, W, c).$ Consider an arbitrary $k$-tuple $T = (v_1, \dots, v_k)$ of vertices in $G$ and let $A$ denote the set of their common neighbours. By assumption, $|A| \geq \frac{C}{2} \cdot \frac{\log n}{p}.$ Hence, by \Cref{lemma:connection}~part~\ref{large_subset}, there exists an edge $e \in E(H),$ such that for $U = U(T) = \ed^{-1}(e) \cap A,$ we have
\[ |U| = \left|\ed^{-1}(e) \cap A\right| \geq \frac{|A|}{(N_{r, s+1}+1)^r} \geq \frac{10 \log n}{p}, \]
where we took $C$ to be large enough compared to $r,k$.
By Lemma~\ref{lemma:connection}~part~\ref{neighbours_intersecting}, $e$ forms a cover of $\ed(v_1), \dots, \ed(v_k)$ and we set $\phi(v_1, \dots, v_k) = e.$ Finally, we need to show that these covers pairwise intersect. Consider two arbitrary $k$-tuples $T, T'$ of vertices in $G.$ If $U(T) \cap U(T') \neq \emptyset,$ then $\phi(T) = \phi(T')$ so we may assume that $U(T)$ and $U(T')$ are disjoint. Since $|U(T)|, |U(T')| \geq \frac{10 \log n}{p},$ by assumption there exists an edge between $U(T)$ and $U(T').$ By Lemma~\ref{lemma:connection}~part~\ref{neighbours_intersecting}, it follows that $\phi(T) \cap \phi(T') \neq \emptyset,$ finishing the proof.
\end{proof}

We now turn our attention to lower bounds for $\tc_r(G)$ in terms of $\hi_r(k)$. 

\begin{lemma} \label{lemma:lower_bound_simple}
  For any fixed $2 \leq r \leq k,$ there exists a constant $c = c(r, k)$ such that for the random graph $G \sim \calG(n, p)$ with $p < c \left( \frac{\log n}{n}\right)^{\frac{1}{k+1}},$ w.h.p. $tc_r(G) \geq \hi_r(k).$
\end{lemma}
\begin{proof}
  Let $H$ be an $r$-partite $r$-graph with intersecting $k$-covers such that $\tau(H) = \hi_r(k) = s.$ For a $k$-tuple $T$ of edges in $H,$ let $\phi(T)$ denote the fixed $k$-cover of $T,$ given by the definition of the intersecting $k$-covers property. For small enough $c,$ by Lemma \ref{lemma:random_for_lb}, we may assume that in $G \sim \Gnp$ there exists an independent set $I$ of size $m = |E(H)|$ such that no $k+1$ vertices in $I$ have a common neighbour. Let $f_1, \dots, f_m$ denote the edges of $H$ and let $v_1, \dots, v_m$ denote the vertices in $I.$

  We aim to use Lemma \ref{lemma:colouring_lb}, so we need to describe a surjective function $\ed_0 \colon V(G) \rightarrow E(H)$ such that $\ed_0(u) \cap \ed_0(v) \neq \emptyset$ when $uv \in E(G).$ Set $\ed_0(v_i) = f_i,$ for $i \in [m],$ which ensures $\ed_0$ is surjective. Now consider $u \in V(G) \setminus I.$ By assumption $u$ has at most $k$ neighbours in $I,$ so let $T = T(u)$ be a $k$-tuple of vertices in $I$ containing all of its neighbours in $I.$ Set $\ed_0(u) = \phi(\ed_0(T)).$

    Let us verify that $\ed_0$ satisfies the conditions required by Lemma \ref{lemma:colouring_lb}. Trivially, $\ed_0$ is surjective as it is a bijection when restricted to $I\subseteq V(G).$ For any $v_i \in I, u \not\in I,$ such that $v_iu \in E(G),$ by construction, $\ed_0(u)$ and $f_i = \ed_0(v_i)$ intersect. By definition, for any $u_1, u_2 \in V(G) \setminus I,$ $\ed_0(u_1) = \phi(\ed_0(T(u_1)))$ and $\ed_0(u_2) = \phi(\ed_0(T(u_2)))$ intersect. Recall that $I$ is an independent set, so there are no more edges to consider. Therefore, $\ed_0$ satisfies the required conditions and the result follows by Lemma \ref{lemma:colouring_lb}.
\end{proof}

\begin{rem}
  This lemma tells us that in order to obtain a lower bound of $\tc_r(\Gnp) \geq s$ one needs to construct an $r$-partite $r$-uniform hypergraph $H$ with cover number at least $s$ and in which we can fix a cover for any $k$-tuple of edges in such a way that all these fixed covers intersect. Technically speaking, the definition of intersecting $k$-covers insists that our fixed $k$-covers all need to be edges of our hypergraph, but provided this is not the case we may simply consider a new hypergraph $H'$ consisting of $H$ together with all the fixed covers added to it. Now, given a $k$-tuple $T$ of edges of our new hypergraph, we can pick any $k$-tuple $T'$ containing all edges of $T$ which were already present in $H$ and let $\phi(T')$ be the fixed cover of $T$ in $H'.$ Because all the new edges pairwise intersect (all our fixed covers are assumed to pairwise intersect), $\phi(T')$ covers all the edges in $T.$
\end{rem}

Lemmas~\ref{lemma:upper_bound_simple}~and~\ref{lemma:lower_bound_simple} prove our connection theorem, \Cref{thm:upto1}. In other words, for fixed $r, t$ the threshold value of $p$ for having $\tc_r(G) \leq t$ is between $\left({\log n}/{n}\right)^{1/\kc_r(t)}$ and $\left({\log n}/{n}\right)^{1/(\kc_r(t)+1)},$ where $\kc_r(t) = \min \{k \in \mathbb{N} \; \vert \; \hi_r(k) \leq t\}.$

Now suppose for some fixed $r,t$ we are able to exactly determine $\kc_r(t) = k$ and we would like to know the precise threshold for having $\tc_r(G) \leq t.$ Let $p \gg \left({\log n}/{n}\right)^{1/k}.$ As argued before, we can assume that $H = H(G, c)$ is of bounded size and then for any set of $k$ vertices in $G,$ in their common neighbourhood we can find a set $U$ of $\Omega(np^k) \gg \log n$ vertices which all map to the same edge in $H.$ However, in order to guarantee an edge between any pair of such sets $U$ (needed to establish the intersecting $k$ covers property) we need $\Omega(np^k) \gg \frac{\log n}{p},$ which precisely means that when $p$ is in the range left open by \Cref{thm:upto1} we can not guarantee that the $k$-covers intersect. So if one insists on finding \emph{exact} thresholds one needs to work a bit harder. To this end, consider some $k-1$ vertices in $G$ and let $U'$ be the set of their common neighbours which have at least one edge towards our set $U.$ It is not difficult to show that w.h.p. $|U'| = \Omega(np^k \cdot |U|) \gg \log^2 n$ and that $U'$ can play a similar role as $U$ did in our previous arguments. Iterating this argument until we find big enough sets to allow us to deduce there must be an edge between any pair of them, motivates the following definition.

\begin{defn*}
  Let $k, m, r$ be positive integers. We say that an $r$-partite $r$-graph $H_0$ is \emph{$(k, m)$-coverable} if there exist $r$-partite $r$-graphs $H_0 \supseteq H_1 \supseteq \dots \supseteq H_m$ such that the following holds:
  \begin{enumerate}[label=P\arabic*)]
    \item \label{cover_in_Hm} For any $k-1$ edges in $E(H_0)$ there exists an edge $e \in E(H_m)$ which intersects each of these $k-1$ edges.
    \item \label{cover_in_next} For any edges $e_1, \dots, e_{k-1} \in E(H_0),$ and any edge $e_k \in H_i, \, 0 \leq i \leq m-1,$ there exists an edge $f \in E(H_{i+1})$ which intersects all edges $e_1, \dots, e_k.$
    \item \label{Hm_all_intersect} All edges in $E(H_m)$ pairwise intersect.
  \end{enumerate}
\end{defn*}

Let us try to further explain what is going on here. First observe that if we set $m=1$ we recover the $k$-cover intersecting property where $H_1$ is simply the subset of the edges which we fix as a cover for some $k$-edges. Note that if $H$ is $(k,m)$-coverable, then it is also $(k-1, m)$-coverable. Furthermore, $(k,m-1)$-coverability implies $(k,m)$-coverability, since if $H_0 \supseteq H_1 \supseteq \dots \supseteq H_{m-1}$ satisfy the properties for $(k,m-1)$ coverability, then $H_0 \supseteq H_1 \supseteq \dots \supseteq H_{m-1} = H_m$ establish $(k,m)$-coverability. So while as we discussed above we are not able to establish that for $p \gg \left(\frac{\log n}{n}\right)^{1/k}$ $\Gnp$ has intersecting $k$-covers (i.e. is $(k,1)$-coverable), we will be able to establish $(k,m)$-coverability for some larger $m$.

Let us show why in a further attempt to illustrate where the definition comes from. $H_0$ will be our auxiliary hypergraph $H(G,W,c)$ so its edges correspond to vertices of $G$ (to be exact, to monochromatic components defined by vertices of $G$). \ref{cover_in_Hm} corresponds to the fact that any $k-1$ vertices in $G$ have $\Omega(np^{k-1}) \gg \frac{\log n}{p}$ many common neighbours which map to the same edge of $H(G,W,c)$. All edges we encounter this way we place in $H_m$. A key idea here is that since the set of vertices in $G$ which all map to this edge is large we know that there needs to be an edge in $G$ between any pair of such sets, meaning that any pair of edges in $H_m$ need to intersect, in other words \ref{Hm_all_intersect} holds. So far we have just recovered our previous argument, however taking this perspective leads us to split the edges of $H_0$ more finely into ``levels'' depending on how many times we can guarantee they appear. $H_m$ consist of edges which appear the most times and $H_0$ takes all edges without any guarantee on how many times they appear. $H_i$ will consist of edges for which we can guarantee they appear at least about $(\log n)^i$ many times. For example, $H_1$ will arise by taking any set of $k$ vertices in $G$ (which correspond to $k$ edges in $H_0$) and then taking the subset $U$ of their common neighbourhood consisting of vertices which map to the same, most frequent, edge of $H_0$. We know $U$ has size at least $\Omega(np^k) \gg \log n$ so we can place all such edges in $H_1$. Now looking at common neighbours of arbitrary $k-1$ vertices of $G$ (so $k-1$ edges in $H_0$) and some vertex in $U$ (corresponding to an edge in $H_1$) we find a subset of them $U'$ of size $|U'| = \Omega(np^k \cdot |U|) \gg \log^2 n$ which all map to the same edge of $H_0$, so all these edges can be placed in $H_2$. Repeating this until the sets become large enough will establish \ref{cover_in_next}. We will establish this more formally as \Cref{lemma:ub_levels}, but let us first introduce some notation and a few ingredients which should also help further familiarize the reader with the $(k,m)$-coverability notion.

For positive integers $r \leq k$ and $m,$ we define $\hi_r(k, m)$ to be the largest possible cover number of a $(k,m)$-coverable $r$-partite $r$-graph. Since $H$ being $(k,1)$-coverable is equivalent to $H$ having intersecting $k$-covers we have $\hi_r(k, 1) = \hi_r(k).$ For $r \leq k,$ we define $\Hi_r(k) = \max_{m \in \mathbb{N}} \hi_r(k, m),$ so the largest possible cover number of an $r$-partite $r$-graph which is $(k,m)$-coverable for some $m$. Since, as we already observed, $(k,m-1)$ coverability implies $(k,m)$-coverability, we know $\hi_r(k,m-1) \le \hi_r(k,m)$ and in particular $\hi_r(k)=\hi_r(k,1)\le \Hi_r(k)$. Similarly since, as we already observed, $(k,m)$-coverability implies $(k-1,m)$-coverability, we know $\hi_r(k,m) \le \hi_r(k-1,m)$ and in particular $\Hi_r(k) \le \Hi_r(r)$, for any $k \ge r$. Finally, note that if an $r$-partite $r$-graph is $(k, m)$-coverable for some $m$, then it has intersecting $(k-1)$-covers, or equivalently it is $(k-1,1)$-coverable. This implies that $\hi_r(k-1) \ge \Hi_r(k).$ This together with our observation that $\hi_r(k)$ is bounded for any $k \ge r$ implies $\Hi_r(k)$ is bounded for any $k > r$. The following lemma shows this holds for $k=r$ as well and nicely illustrates how one can use the notion of $(k, m)$-coverability to bound the cover number of an $r$-partite $r$-graph.
\begin{lemma}\label{lem:bounded}
	For any $r \geq 2,\, \Hi_r(r) \leq r^2.$
\end{lemma}
\begin{proof}
	Let $H$ be an $(r,m)$-coverable $r$-partite $r$-graph, for some $m.$ We need to show that $\tau(H) \leq r^2.$ Let $H = H_0 \supseteq H_1 \supseteq \dots \supseteq H_m$ certify that $H$ is $(r,m)$-coverable. Let $e$ be an arbitrary edge in $H_m.$ Assume that $e$ is not a cover of $H,$ as otherwise $\tau(H) \leq r,$ and define $0 \leq \ell < m$ to be the maximum index $i$ such that $H_i$ contains an edge disjoint from $e.$ Let $f \in H_\ell$ be such an edge. Suppose in $H$ there is a matching of size $r+1$ containing $e, f$ and consider the $r$ edges different from $e$ in this matching. By assumption, these $r$ edges have a transversal cover $e' \in H_{\ell + 1}.$ However, $e'$ covers a matching of size $r$ so it cannot intersect $e,$ a contradiction. Therefore, there exists a maximal matching of size at most $r$ in $H,$ implying $\tau(H) \leq r^2.$
\end{proof}

We note that this upper bound on $\Hi_r(r)$ is the key reason behind  why \Cref{prop:begin-range} holds, at least once we translate it to our ``hypergraph'' setting. Formalising our previous discussion we prove the following ``translation'' lemma which strengthens \Cref{lemma:upper_bound_simple} and which together with \Cref{lem:bounded} will prove \Cref{prop:begin-range}. 

\begin{lemma} \label{lemma:ub_levels}
  For any fixed integers $2 \leq r \leq k,$ there exists a constant $C = C(r, k)$ such that for $p > C \left(\frac{\log n}{n} \right)^{1/k},$ the random graph $G \sim \Gnp$ w.h.p. satisfies $\tc_r(G) \leq \Hi_r(k).$
\end{lemma}
\begin{proof}
  We will prove in subsequent \Cref{lem:bounded} that $\Hi_r(k) \le \Hi_r(r)$ is bounded, so let us denote by $s = \Hi_r(k)$. Let $C$ be large enough and assume $G$ satisfies the properties given by Lemmas \ref{lemma:no_large_empty_bipartite}, \ref{lemma:common_neigh_k_1} and \ref{lemma:large_neigh}. Suppose $\tc_r(G) \geq s + 1$ and let $c$ be an $r$-edge-colouring of $G$ such that $G$ cannot be covered by $s$ monochromatic components. By Lemma \ref{lemma:connection}, $\tau(H(G, c)) \geq s + 1$ and there is a subset of vertices $W \subseteq V(G)$ of size at most $N_{r, s+1}$ such that $\tau(H(G, W, c)) \geq s + 1.$ Let $(H, \ed, \vt) = (H, \ed, \vt)(G, W, c)$ and let $H_0 = H.$ Let $m$ be the minimum integer such that $(10 \log n)^{m-1} \geq \frac{1}{p}.$ We will show that $H_0$ is $(k, m)$-coverable, thus contradicting $\Hi_r(k) = s.$

  Define 
  \[ n_i = \begin{cases}
  \left( 10 \log n \right)^i, &\text{for } 0 \leq i \leq m-2,\\
  \frac{1}{p}, &\text{for } i=m-1,\\
  \frac{10 \log n}{p}, &\text{for } i=m.
  \end{cases}
  \]

  Recall that $E(H_0) = \{ \ed(u) \, \vert \, u \in V(G) \}.$ For $1 \leq i \leq m,$ we define 
  \[ H_i = \left\{ e \in E(H_0) \; \big\vert \; |\ed^{-1}(e)| \geq n_i \right\}. \]

  We show that $H_0, \dots, H_m$ certify that $H$ is $(k, m)$-coverable. Trivially, $H_0 \supseteq H_1 \supseteq \ldots \supseteq H_m.$ Consider any $k-1$ edges $e_1 = \ed(v_1), \dots, e_{k-1} = \ed(v_{k-1})$ in $H$ and let $A = N(\{v_1\}, \dots, \{v_{k-1}\}).$ By the assumed property given by Lemma \ref{lemma:common_neigh_k_1}, $|A| \geq \frac{C}{2} \frac{\log n}{p}.$ For large enough $C,$ by Lemma~\ref{lemma:connection}~part~\ref{large_subset}, there exists $f \in E(H_0)$ such that
	\[ |\ed^{-1}(f)| \geq \left|\ed^{-1}(f) \cap A\right| \geq \frac{\frac{C}{2} \frac{\log n}{p}}{(N_{r, s+1} + 1)^r} \geq \frac{10 \log n}{p}. \]
	Thus, $f \in E(H_m),$ and by Lemma~\ref{lemma:connection}~part~\ref{neighbours_intersecting}, $f$ intersects the edges $e_1, \dots, e_{k-1}.$ Hence, Property \ref{cover_in_Hm} is satisfied.

	The proof of Property \ref{cover_in_next} goes along similar lines. Let $e_1, \dots, e_{k-1}$ be as before and consider an arbitrary edge $e_k \in E(H_i),$ where $0 \leq i \leq m-1.$ Let $A = N\big(\ed^{-1}(e_k), \{v_1\}, \dots, \{v_{k-1}\}\big).$ By definition of $H_i$ and by the assumed property from Lemma \ref{lemma:large_neigh}:
\[ |A| \geq \frac{C \log n}{6} \cdot |\ed^{-1}(e_k)| \geq \frac{C \log n}{6} \cdot n_i. \]
	Using Lemma~\ref{lemma:connection}~part~\ref{large_subset}, we can find $f \in E(H_0)$ such that 
  \[ |\ed^{-1}(f) \cap A| \geq \frac{\frac{C \log n}{6} \cdot n_i}{(N_{r, s+1}+1)^r} \geq (10 \log n) \cdot n_i \geq n_{i+1}. \]
  Thus, $f \in E(H_{i+1}),$ and $f$ covers the edges $e_1, \dots, e_k$ by Lemma~\ref{lemma:connection}~part~\ref{neighbours_intersecting}, so Property \ref{cover_in_next} holds.

Finally, consider arbitrary edges $e_1, e_2 \in E(H_m).$ By definition of $H_m$ and our assumption given by Lemma \ref{lemma:no_large_empty_bipartite}, there exists an edge between $\ed^{-1}(e_1)$ and $\ed^{-1}(e_2)$ in $G.$ Hence, $e_1$ and $e_2$ intersect by Lemma~\ref{lemma:connection}~part~\ref{neighbours_intersecting}, implying Property \ref{Hm_all_intersect}.

We conclude that $H_0$ is $(k, m)$-coverable, finishing the proof.
\end{proof}

\begin{rem}
  In the above argument the $m$ we work with is unbounded with $n$ so we really require an upper bound on $\tau(H)$ for a $(k,m)$-coverable $H$ for \emph{all} $m$. That said, if one only has the information about a fixed choice of $m$ the above argument can easily be modified to give the following slightly weaker bound on $p$. Suppose $t, r, k, m$ are positive integers such that any $(k, m)$-coverable $r$-partite $r$-graph satisfies $\tau(H) \leq t.$ Then, there is a positive constant $C$ such that if $(np^k)^m > \frac{C \log n}{p},$ then w.h.p. $G \sim \Gnp$ satisfies $\tc_r(G) \leq t.$
\end{rem}

\noindent Combining the previous two lemmas proves \Cref{prop:begin-range}.

The last result of this section is the lower bound counterpart to \Cref{lemma:ub_levels}. It shows that one can use a $(k,m)$-coverable hypergraph with a large cover number to prove lower bounds on $\tc_r(G).$ It gives weaker bounds than \Cref{lemma:lower_bound_simple} but requires a weaker condition. We only state it here for completeness, since we are not going to use it in this paper and point an interested reader to \cite{thesis} for its proof.

\begin{proposition} \label{prop:lb_levels}
	Let $r \ge k \geq 2$ be integers and suppose $\Hi_r(k) \geq s.$ Then, there exists $\varepsilon = \varepsilon(r, k, s) > 0$ such that for $p = n^{-\frac{1}{k} + \varepsilon},$ w.h.p. $G \sim \Gnp$ satisfies $\tc_r(G) \geq s.$
\end{proposition}

Let us briefly summarize the above results. Suppose for some fixed $r,t$ we want to determine the threshold for having $\tc_r(\Gnp) \leq t.$ Determining $\kc_r(t) = k$ places the threshold between $\left({\log n}/{n}\right)^{1/k}$ and $\left({\log n}/{n}\right)^{1/(k+1)}$ by \Cref{thm:upto1}. We may attempt to close this gap by studying $\Hi_r(k).$ If $\Hi_r(k) \leq t,$ we close the gap and the threshold is precisely $\left({\log n}/{n}\right)^{1/k}$ by \Cref{lemma:ub_levels}. This is precisely what happens in case $r=t=3$ and leads to our \Cref{thm:three_threshold}. It remains possible that precisely this way one can determine the threshold for $\tc_r(\Gnp) \leq t$ whenever $r=t,$ although even for $r=3, t=4,$ this method fails. In such a case, that is, when $\Hi_r(k) > t,$ we can improve the lower bound on the threshold by a factor of $n^{\varepsilon}$ through \Cref{prop:lb_levels}. In fact in this case, using the remark following \Cref{lemma:ub_levels} one can also obtain an improved upper bound, with the improvement depending on the minimum $m$ for which there exists a $(k,m)$-coverable $r$-graph with cover number larger than $t.$

From the definitions, we know that $\hi_r(k-1) \geq \Hi_r(k) \geq \hi_r(k).$ Therefore, in general to determine the threshold for $\tc_r(G) \leq t,$ one should determine $\kc_r(t)$ exactly and only then it makes sense to study $(k,m)$-coverable hypergraphs, in particular, to determine whether $\Hi_r(\kc_r(t)) \leq t.$

\section{Three colours}
In this section we demonstrate how to use the tools we developed in the preceding section to prove \Cref{thm:three_threshold}. 

The lower bound part of Theorem \ref{thm:three_threshold} follows from the example of Ebsen, Mota and Schnitzer (see \cite{kohayakawa2019monochromatic}). Their example can actually be easily translated into our language as follows. We want to find a $3$-partite $3$-uniform hypergraph $H$ which has $\tau(H) \ge 4$ but for any three edges we can fix a cover in such a way that they all intersect.  We claim that taking $H$ to consist of four disjoint edges suffices (and after translation precisely matches the construction of Ebsen, Mota and Schnitzer). To see this, we present a construction which matches Figure~\ref{fig:auxilliary}. Let $r_1,b_1,g_1$ be three vertices belonging to different parts and different edges in $H$ and denote the remaining edge as $r_2b_2g_2$ (where, w.l.o.g. we assume $r_1,r_2$ belong to the same part, as well as $b_1,b_2$ and $g_1,g_2$). We can take as our fixed covers $r_1b_1g_1,r_1b_1g_2,r_1b_2g_1,r_2b_1g_1$. Note that any pair of them intersect, if we drop any edge from $H$ one of these covers covers the remaining edges, and since $H$ consists of four disjoint edges, $\tau(H) \ge 4$. This shows\footnote{Technically, to have intersecting $3$-covers, we require the covers to be edges, but as per remark after \Cref{lemma:lower_bound_simple} we can always simply add the covers to $H$ in order to ensure this holds.} $\kc_3(3) \ge 4$ . 
Hence, by Theorem \ref{thm:upto1} part a), the threshold is indeed at least $(\log n /n )^{1/4}$.

The simplicity of the example in our setting showcases the power of thinking about the tree covering problem in the hypergraph setting. It also might lead one to believe that it can not possibly be tight, since there is no interaction between the edges. Another reason towards this is that there are plenty more complicated examples establishing the same bound. This intuition indeed turned out to be true already for $r\ge 4$, as was shown in \cite{bucic2019covering}. Indeed, for the threshold for $\tc_r(\Gnp) \leq r,$ the example above only gives a lower bound matching the threshold that any $r+1$ vertices have a common neighbour, while they found an example showing the edge probability should be at least such that any $\binom{r-1}{\floor{(r-1)/2}}+\binom{r-1}{\ceil{(r-1)/2}}$ vertices are likely to have a common neighbour.

Let us now turn towards proving the upper bound. We begin with some notation. Let $H$ be a $3$-partite $3$-graph. We denote the parts of $H$ by $V_R, V_B, V_G.$ We denote the vertices in part $V_R$ with $r_1, r_2, \dots,$ and analogously we use $b_1, b_2, \dots$ for vertices in part $V_B$ and $g_1, g_2, \dots$ for vertices in part $V_G.$

The following specific hypergraph, which we denote by $H^*$ and which we have actually seen in the above example, plays an important role in our proof.

\begin{minipage}[t]{0.45\textwidth}
\centering
\vspace{-3.3cm}
\begin{align*} \label{special_H} 
  V(H^*) &= \{ r_1, r_2, b_1, b_2, g_1, g_2 \}, \\
  E(H^*) &= \{ r_1b_1g_2, r_1b_2g_1, r_2b_1g_1, r_2b_2g_2 \}.
\end{align*}
\end{minipage}%
\hfill
\begin{minipage}[t]{0.45\textwidth}
\vbox{
  \centering
  \includegraphics[scale=0.8]{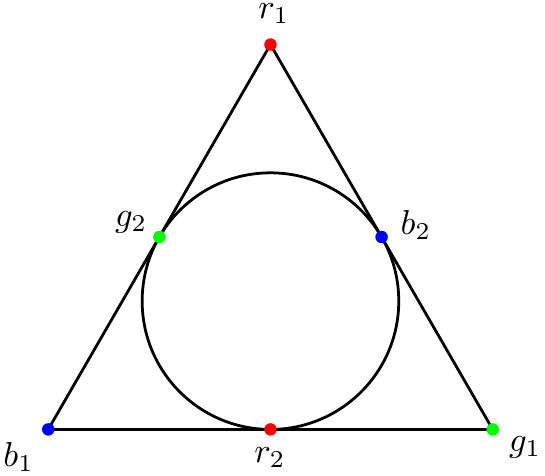}
  \captionof{figure}{$H^*$}
}
\end{minipage}

We state several immediate properties of $H^*.$
\begin{itemize}
  \item Any two edges in $H^*$ intersect.
  \item If an $r$-partite hypergraph $H$ contains a copy of $H^*,$ then any edge of $H$ which intersects each edge of $H^*$ belongs to this copy of $H^*$ in $H.$
  \item $H^*$ is symmetric with respect to edges and with respect to parts.
\end{itemize}
The first property is immediate. For the second one consider an edge $e$ of $H$. If $|e \cap \{r_1,b_1,g_1\}|\ge 2,$ the fact $e$ must intersect $r_2b_2g_2$ implies its third vertex must be indexed by $2$ so it belongs to $H^*$. Otherwise, if it contains $r_1$ it cannot contain $b_1,g_1$ so it cannot intersect $r_2b_1g_1,$ so the only option is that it is equal to $r_2b_2g_2,$ which belongs to $H^*$ as claimed. The last point is immediate and we only use it to simplify our proofs, e.g. if $f$ is an edge which does not intersect some edge $e \in E(H^*),$ we may assume that $e = r_2b_2g_2.$

We begin with a simple characterization of intersecting $3$-partite $3$-graphs.
\begin{lemma} \label{lemma:all_intersecting_structure}
  Let $H$ be a $3$-partite $3$-graph such that all of its edges are pairwise intersecting. Then, $H$ satisfies one of the following:
  \begin{enumerate}[label=\alph*)]
    \item \label{all_intersect} $\bigcap_{e \in E(H)} e \neq \emptyset$; or
    \item \label{two_out_of_three} there exist vertices $r, b, g$ in different parts such that $\forall e \in E(H),$ $|e \cap \{r, b, g\}| \geq 2$; or
    \item \label{special_H} $H$ is isomorphic to $H^*$ up to isolated vertices.
  \end{enumerate}
\end{lemma}
\begin{proof}
  We assume that \ref{all_intersect} and \ref{two_out_of_three} do not hold and show that then \ref{special_H} holds. Clearly $H$ is not empty. Let $e_1 = r_1b_1g_2$ be an edge in $H.$ Since \ref{two_out_of_three} does not hold for the vertex set $\{r_1, b_1, g_2\}$, there exists an edge $e_2 \in E(H)$ intersecting $e_1$ in precisely one vertex. Without loss of generality, we can assume $e_2 = r_1b_2g_1 \in E(H).$ Since not all edges share a vertex, there is some edge $e_3$ not containing $r_1,$ say $r_2 \in e_3.$ Because $e_3$ intersects $e_1$ and $e_2,$ it follows that $e_3 = r_2b_1g_1$ or $e_3 = r_2b_2g_2.$ These cases are symmetric (by relabeling $b_1 \Longleftrightarrow b_2$ and $g_1 \Longleftrightarrow g_2$), so we may assume $e_3 = r_2b_1g_1.$ Because \ref{two_out_of_three} does not hold for $\{r_1, b_1, g_1\},$ there is an edge $e_4$ with at most one vertex in $\{r_1, b_1, g_1\}.$ It is easy to see that the only way such an $e_4$ can intersect $e_1, e_2, e_3$ is if $e_4 = r_2b_2g_2 \in E(H).$ So these $4$ edges make a copy of $H^{*}$. Finally, by the second property of $H^{*}$ observed above, there cannot be an additional edge $e_5$ different from $e_1, e_2, e_3, e_4$ which intersects all the previous edges. Hence, $H$ is isomorphic to $H^*$ up to isolated vertices, as desired.
\end{proof}

\begin{rem}
This lemma immediately implies Ryser's conjecture for intersecting $3$-partite hypergraphs since in each of the three cases it is easy to see the cover number is at most $2$. This was first proved by Tuza in 1983 \cite{tuza}.
\end{rem}

We will use the following simple observation several times in our proof.
\begin{observation}
  \label{obs:konig}
  Let $H$ be a $3$-partite $3$-graph and let $T$ be a set of $t$ vertices, for some $t \leq \tau(H).$ Let $A$ be an arbitrary part of $H.$ Then, there exists a set $F \subseteq E(H)$ of $\tau(H) - t$ edges disjoint from $T$, such that any pair of distinct edges in $F$ intersect only inside part $A$.
\end{observation}
\begin{proof}
  Let $E' \subseteq E(H)$ be the set of edges disjoint from $T.$ Without loss of generality, assume $A = V_R,$ and consider the bipartite graph $G$ with parts $(V_B, V_G)$ and the edge set $E(G) = \{ \left(e \cap V_B, e \cap V_G\right) \; \vert \; e \in E' \}.$ Observe that we can cover $H$ using $T$ and the minimum cover of $G,$ so $\tau(G) \geq \tau(H) - t.$ By the K\H{o}nig-Egerv\'{a}ry theorem, we can find a matching of size $\tau(H) - t$ in $G.$ For every edge in this matching, take one corresponding edge in $H.$ This set of edges is disjoint from $T$ and outside $A$ span a matching, so can only intersect inside $A$.
\end{proof}

We are now ready to prove the final ingredient missing for our proof of \Cref{thm:three_threshold}.

\begin{lemma} \label{lemma:hi34}
  $\Hi_3(4) \leq 3.$
\end{lemma}
\begin{proof}
  Assume $\Hi_3(4) > 3.$ Then, there exist $m \in \mathbb{N}$ and a $(4,m)$-coverable $3$-partite $3$-graph $H$ with $\tau(H) > 3.$ Let $H = H_0, H_1, H_2, \dots, H_m$ be hypergraphs certifying that $H$ is $(4,m)$-coverable. By assumption, the edges of $H_m$ pairwise intersect, so by Lemma \ref{lemma:all_intersecting_structure}, $H_m$ satisfies one of the following:
  \begin{enumerate}[label=\alph*)]
    \item $\bigcap_{e \in E(H_m)} e \neq \emptyset$; or
    \item there exist vertices $r, b, g$ in different parts such that $\forall e \in E(H_m),$ $|e \cap \{r, b, g\}| \geq 2$; or
    \item $H_m$ is isomorphic to $H^*$ up to isolated vertices.
  \end{enumerate}

  We handle these cases in order.
  \begin{enumerate}[label=\alph*)]
    \item Without loss of generality, assume $r_1 \in \bigcap_{e \in E(H_m)} e.$ By Observation \ref{obs:konig}, there exist three edges $e_1, e_2, e_3$ not containing $r_1$ such that $e_i \cap e_j \subseteq V_R, \, 1 \leq i < j \leq 3.$ However, these three edges cannot have a transversal cover containing $r_1,$ so in particular can't be covered by an edge of $H_m$, contradicting property \ref{cover_in_Hm}.
    \item Let $r_1, b_1, g_1$ be vertices such that $|e \cap \{r_1, b_1, g_1\}| \geq 2, \forall e \in E(H_m).$ Let $0 \leq \ell < m$ be the maximum index $i$ such that there exists an edge $e \in E(H_i)$ satisfying $|e \cap \{r_1, b_1, g_1\}| \leq 1.$ As $\tau(H_0) > 3,$ there exists such an edge in $E(H_0),$ so $\ell$ is well-defined. Without loss of generality, we may assume there exists an edge $e = r_ib_jg_k \in E(H_\ell)$ with $j, k \neq 1.$ We consider two cases:
    \begin{itemize}
      \item $i = 1.$ Observe that by the maximality of $\ell$ any edge $f$ in $E(H_{\ell + 1})$ intersecting $e$ must contain $r_1$ (otherwise it does not contain $r_1$ so is forced to contain $b_1,g_1$ and so indeed does not intersect $e$). By Observation \ref{obs:konig}, there exist three edges $e_1, e_2, e_3 \in E(H_0)$ not containing $r_1$ such that $e_i \cap e_j \subseteq V_R, \, 1 \leq i < j \leq 3.$ By assumption there exists an edge $f \in E(H_{\ell+1})$ covering $e, e_1, e_2, e_3.$ Because $f$ intersects $e,$ we have $r_1 \in f$ by the above claim. However, then $f$ can only cover two out of $e_1, e_2, e_3,$ since they do not contain $r_1$ and restricted to $V_B \cup V_G$ span a matching. This is a contradiction.
      \item $i \neq 1.$ Since $\tau(H) > 3,$ there exists an edge $e_1 \in E(H)$ \emph{disjoint} from $\{r_i, b_1, g_1\};$ an edge $e_2$ \emph{disjoint} from $\{r_1, b_j, g_1\}$ and an edge $e_3$ \emph{disjoint} from $\{r_1, b_1, g_k\}.$ The edges $e, e_1, e_2, e_3$ cannot be covered by an edge with two vertices in $\{r_1, b_1, g_1\},$ since the third vertex would need to be $r_i,b_j$ or $g_k$ and each of the three cases is precisely excluded by one of $e_1,e_2$ or $e_3$. This is once again a contradiction.
    \end{itemize}
    \item $E(H_m) = E(H^*) = \{f_1, f_2, f_3, f_4\}.$ Let $\ell, \, 0 \leq \ell < m$ be an index such that $H_{\ell} \neq H^*$ and $H_{\ell+1} = H^*.$ From $\tau(H) > 3$ it follows that $H_0 \neq H^*,$ so $\ell$ is well-defined. Let $e \in H_{\ell} \setminus H^*.$ By the second property of $H^*,$ there exists some $f \in E(H^*)$ disjoint from $e.$ Without loss of generality, $e \cap f_4 = \emptyset.$ Since $\tau(H) > 3,$ there are edges $e_1, e_2, e_3 \in E(H_0),$ such that $e_i \cap f_i = \emptyset, \, 1 \leq i \leq 3.$ The edges $e, e_1, e_2, e_3$ then do not have a cover in $H_{\ell+1} = H^*,$ a contradiction.
    \end{enumerate}
    \vspace{-0.9 cm}
\end{proof}

Combining Lemmas~\ref{lemma:ub_levels}~and~\ref{lemma:hi34} we obtain the upper bound in \Cref{thm:three_threshold} and hence complete its proof.

\section{Concluding remarks and open problems}
In this paper we continue the study of the function $\tc_r(\Gnp)$ defined as the minimum number of monochromatic components needed to cover all vertices of $G \sim \Gnp$ for any $r$-colouring of its edges. Our main contribution is a further development of a connection, first observed in \cite{bucic2019covering}, between this function and the following type of natural local to global, Helly-type problems for hypergraphs first introduced by Erd\H{o}s, Hajnal and Tuza \cite{tuza-helly} about 30 years ago. 

Given an $r$-uniform hypergraph $H$ in which we know that any $k$ edges have a cover of size $r$ how large can the cover number of $H$ be? Proving upper bounds for this question translate to upper bounds on the threshold for $\tc_r(\Gnp) \le t$ with an appropriate relation between the parameters.
This was observed in \cite{bucic2019covering} where they also show that one can consider a slightly easier hypergraph problem, namely it is enough to restrict attention to $r$-partite hypergraphs and one may assume that the covers of any $k$ edges are transversals. While this allowed them to obtain slightly stronger upper bounds the main advantage is that the connection in this form goes the other way as well, namely if one has a construction for this $r$-partite variant of the hypergraph problem then this leads to a lower bound for the threshold problem. Unfortunately, there is some loss involved (essentially one needs one extra colour). 

Our new result allows one to consider an even easier problem, namely we may assume that all the covers of $k$ edges intersect and are actual edges of the hypergraph, leading to the following question.

\begin{question}\label{qn:int-covers}
  Let $H$ be an $r$-partite $r$-graph in which we can fix a transversal cover for any $k$ edges, which is an edge of $H$ itself and with the property that any two such covers intersect. What is the maximum value of $\tau(H)$? 
\end{question}

We show that the smallest value of $k$ which guarantees $\tau(H) \le t$, which we denote by $\kc_r(t)$, controls very tightly the threshold for $\tc_r(\Gnp)\le t$. In particular, the threshold is in-between the point when any $\kc_r(t)$ vertices in $\Gnp$ have a common neighbour and the point when any $\kc_r(t)+1$ vertices have a common neighbour.

There are two main benefits of our new approach. It makes it much easier to prove upper bounds on the thresholds and makes it possible to prove much tighter bounds. In particular, this allows one to shift attention to the hypergraph problem which seems much easier to handle, as evidenced both by results in \cite{bucic2019covering} and our new improvements.

Theorem 6.1 in \cite{bucic2019covering} together with a relation between our parameter $\kc_r$ and their parameter $\text{hp}_r$ determines the answer to \Cref{qn:int-covers} up to a logarithm factor, namely it is between $\Omega\left(\frac{r^2}{\log k}\right)$ and $O\left(\frac{r^2 \log r}{\log k}\right)$.  It would be very interesting to determine whether this $\log r$ term is necessary or not. 

Another question of particular interest, since it translates to determining the threshold for when $\tc_r(\Gnp) \le r$, is to determine $\kc_r(r)$ more precisely. Similarly from Theorem 6.1 in \cite{bucic2019covering} we can get $\Omega(2^r/\sqrt{r})\le \kc_r(r) \le 2^r.$ Here, given the results for small $r$ it seems far more likely the lower bound is closer to the truth and our easier variant of the problem could be helpful in showing this.

Finally, it could be interesting to find a hypergraph cover problem whose solution precisely determines the threshold for $\tc_r(\Gnp) \le t$. We have shown that a further refinement is possible through considering level hypergraphs and this was instrumental in determining the correct threshold for $\tc_3 (\Gnp) \le 3$. The key obstacle here is our limited understanding of the following nice random graph problem.

Given two graphs $G, F$ we say that a function $f \colon V(G) \rightarrow V(F)$ is a graph \textit{homomorphism} from $G$ to $F$ if $f(u)f(v) \in E(F)$ whenever $uv \in E(G).$  
For a hypergraph $H,$ its intersection graph, denoted by $\intersection(H),$ is the graph whose vertices are the edges of $H$ and two vertices are joined by an edge whenever the corresponding edges intersect. In this definition we allow self-loops so for any non-empty edge of $H,$ the corresponding vertex in $\intersection(H)$ has a loop.

\begin{question} \label{question:int_graph}
  Given a fixed graph $F,$ determine the threshold $p = p(F)$ for the existence of a surjective homomorphism from the random graph $G \sim \Gnp$ to $F.$
\end{question}
If we were able to determine this threshold for all graphs $F,$ then we can essentially find the threshold for $\tc_r(G) \leq t,$ for any fixed values of  $r$ and $t.$ Indeed, this threshold is equal to $\max_H p(\intersection(H)),$ where $H$ ranges over all of the, finitely many, hypergraphs with parts of size $N_{r, t+1}$ satisfying $\tau(H) > t.$ To see this, recall that by Lemma~\ref{lemma:colouring_lb}, a surjective homomorphism from $G$ to $\intersection(H)$ provides a colouring of $G$ for which at least $\tau(H)$ monochromatic components are needed to cover the vertices of $G$. On the other hand, Lemma~\ref{lemma:connection} parts \ref{small_subgraph} and \ref{neighbours_intersecting} imply that if at least $t+1$ monochromatic components are needed to cover $G,$ then there exists a surjective homomorphism from $G$ to $\intersection(H),$ for some $r$-partite $r$-graph $H$ with cover number at least $t+1$ and at most $N_{r, t+1}$ vertices in each part.

Of course, Question \ref{question:int_graph} might be significantly easier if we restrict $F$ to a certain class of graphs. Some of the results of Section \ref{sec:connection} can be viewed as bounds for $p(F)$ when $F$ is an intersection graph of an $r$-partite $r$-graph with a large cover number.

Let us now turn back to the specific application of our new connection results. For the $3$-coloured case, we have determined the precise threshold for $\tc_3(G) \leq 3$. For the only remaining range, namely if $(\log n/ n)^{1/3} \ll p \ll (\log n/ n)^{1/4},$ we know by our \Cref{thm:three_threshold} that $\tc_3(\Gnp) \ge 4$ and a result from \cite{bucic2019covering} tells us it is finite, in particular they show $\tc_3(\Gnp) \le 21$. Our methods allow us to show it is always at most $5$ and that this is tight for a while (see \cite{thesis} for the details). This only leaves open the question of the threshold $p^*$ for $\tc_3(G) \leq 4.$ Somewhat unexpectedly, this threshold does not match a threshold for any $k$ vertices to have a common neighbour in $\Gnp$. Our methods allow one to show that $n^{-1/3 + \varepsilon} < p^* < n^{-2/7 + o(1)},$ for some constant $\varepsilon > 0$ (see \cite{thesis} for the details). It is probably not hard to further improve either of these exponents, but it seems additional ideas are needed to determine the precise value of $p^*.$ We summarize the behaviour of $\tc_3(G)$ for different ranges of $p$ in the following table.

\begin{center}
    \begin{tabular}{c|c|c|c|c|c}
        $p$ & $\left[0, \left(\frac{\log n}{n}\right)^{1/3}\right)$ & $\left(\left(\frac{\log n}{n}\right)^{1/3}, p^*\right)$ & $\left(p^*,\left(\frac{\log n}{n}\right)^{1/4}\right)$ & $\left(\left(\frac{\log n}{n}\right)^{1/4}, 1 - o(1)\right)$ & $\big(1 - o(1), 1 \big]$ \\
        \hline
        $\tc_3(G)$ & $\infty$ & $5$ & $4$ & $3$ & $2$
    \end{tabular}
\end{center}

\textbf{Remark.} Some of the work in this paper has been done during the master thesis project of the first author \cite{thesis} at ETH Z\"urich. In order to focus on the most interesting aspects of our new ideas we do not include here some of the minor or more technical side results, but we often mention them to give a more complete picture.

\textbf{Acknowledgments}
We would like to thank Benny Sudakov for useful conversations and comments regarding this paper and Shoham Letzter for useful suggestions and a careful reading of an early version of the paper. We are also very grateful to the anonymous referees for their very helpful comments and suggestions. 


\begin{thebibliography}{10}

\bibitem{balpartitioning}
D.~Bal and L.~DeBiasio, \emph{{Partitioning random graphs into monochromatic
  components}}, Electronic Journal of Combinatorics \textbf{24} (2017), P1.18.

\bibitem{bennett2019large}
P.~Bennett, L.~DeBiasio, A.~Dudek, and S.~English, \emph{Large monochromatic
  components and long monochromatic cycles in random hypergraphs}, European
  Journal of Combinatorics \textbf{76} (2019), 123--137.

\bibitem{bollobas1965generalized}
B.~Bollob{\'a}s, \emph{{On generalized graphs}}, Acta Mathematica Hungarica
  \textbf{16} (1965), no.~3-4, 447--452.

\bibitem{thesis}
D.~Brada{\v{c}}, \emph{Covering random graphs by monochromatic trees}, Master
  thesis, \url{https://doi.org/10.3929/ethz-b-000498462}, 2021.

\bibitem{bucic2019covering}
M.~Buci\'{c}, D.~Kor\'{a}ndi, and B.~Sudakov, \emph{{Covering graphs by
  monochromatic trees and Helly-type results for hypergraphs}}, Combinatorica
  \textbf{41} (2021), no.~3, 319--352.

\bibitem{conlon2016combinatorial}
D.~Conlon and W.~T. Gowers, \emph{Combinatorial theorems in sparse random
  sets}, Annals of Mathematics (2016), 367--454.

\bibitem{erdos1992small}
P.~Erd\H{o}s, D.~G. Fon-Der-Flaass, A.~V. Kostochka, and Z.~Tuza, \emph{Small
  transversals in uniform hypergraphs}, Siberian Advances in Mathematics
  \textbf{2} (1992), no.~1, 82--88.

\bibitem{erdHos1991local}
P.~Erd{\H{o}}s, A.~Hajnal, and Z.~Tuza, \emph{{Local constraints ensuring small
  representing sets}}, Journal of Combinatorial Theory, Series A \textbf{58}
  (1991), no.~1, 78--84.

\bibitem{fon1999transversals}
D.~G. Fon-Der-Flaass, A.~V. Kostochka, and D.~R. Woodall, \emph{{Transversals
  in uniform hypergraphs with property $(7, 2)$}}, Discrete mathematics
  \textbf{207} (1999), no.~1-3, 277--284.

\bibitem{gerencser1967ramsey}
L.~Gerencs{\'e}r and A.~Gy{\'a}rf{\'a}s, \emph{{On Ramsey-type problems}},
  Annales Universitatis Scientiarium Budapestinensis de Rolando E\"{o}tv\"{o}s
  Nominatae, Sectio Mathematica \textbf{10} (1967), 167--170.

\bibitem{gyarfas2016vertex}
A.~Gy{\'a}rf{\'a}s, \emph{{Vertex covers by monochromatic piecesâ€”a survey of
  results and problems}}, Discrete Mathematics \textbf{339} (2016), no.~7,
  1970--1977.

\bibitem{henderson}
J.~R. Henderson, \emph{Permutation decomposition of (0, 1)-matrices and
  decomposition transversals}, Ph.D. thesis, California Institute of
  Technology, 1971.

\bibitem{kohayakawa2019covering}
Y.~Kohayakawa, W.~MendonÃ§a, G.~O. Mota, and B.~SchÃ¼lke, \emph{Covering
  3-edge-colored random graphs with monochromatic trees}, SIAM Journal on
  Discrete Mathematics \textbf{35} (2021), no.~2, 1447--1459.

\bibitem{kohayakawa2019monochromatic}
Y.~Kohayakawa, G.~O. Mota, and M.~Schacht, \emph{Monochromatic trees in random
  graphs}, Mathematical Proceedings of the Cambridge Philosophical Society
  \textbf{166} (2019), no.~1, 191--208.

\bibitem{konig}
D.~K{\H{o}}nig, \emph{Gr{\'a}fok {\'e}s m{\'a}trixok}, Matematikai {\'e}s
  Fizikai Lapok \textbf{38} (1931), 116--119.

\bibitem{cycle-cover}
D.~Kor{\'a}ndi, F.~Mousset, R.~Nenadov, N.~{\v{S}}kori{\'c}, and B.~Sudakov,
  \emph{Monochromatic cycle covers in random graphs}, Random Structures \&
  Algorithms \textbf{53} (2018), no.~4, 667--691.

\bibitem{kostochka2002transversals}
V.~Kostochka, \emph{{Transversals in uniform hypergraphs with property $(p,
  2)$}}, Combinatorica \textbf{22} (2002), no.~2, 275--285.

\bibitem{lang}
R.~Lang and A.~Lo, \emph{{Monochromatic cycle partitions in random graphs}},
  Combinatorics, Probability and Computing \textbf{30} (2021), no.~1, 136--152.

\bibitem{lovasz-covers}
L.~Lov\'{a}sz, \emph{On minimax theorems of combinatorics}, Mathematikai Lapok
  \textbf{26} (1975), no.~3-4, 209--264 (1978). \MR{510823}

\bibitem{schacht2016extremal}
M.~Schacht, \emph{Extremal results for random discrete structures}, Annals of
  Mathematics (2016), 333--365.

\bibitem{tuza}
Z.~Tuza, \emph{Ryser's conjecture on transversals of {$r$}-partite
  hypergraphs}, Ars Combinatoria \textbf{16} (1983), no.~B, 201--209.

\bibitem{tuza-helly}
Z.~Tuza, \emph{Helly property in finite set systems}, Journal of Combinatorial
  Theory, Series A \textbf{62} (1993), no.~1, 1--14.

\end{thebibliography}

\providecommand{\bysame}{\leavevmode\hbox to3em{\hrulefill}\thinspace}
\providecommand{\MR}{\relax\ifhmode\unskip\space\fi MR }
\providecommand{\MRhref}[2]{%
  \href{http://www.ams.org/mathscinet-getitem?mr=#1}{#2}
}
\providecommand{\href}[2]{#2}

\end{document}